\documentclass[reqno,11pt]{amsart}
\usepackage{a4wide,xcolor,eucal,enumerate,mathrsfs}
\usepackage[normalem]{ulem}
\usepackage{amsmath,amssymb,epsfig,amsthm} 
\usepackage[latin1]{inputenc}
\usepackage{psfrag}          

\numberwithin{equation}{section}

\newtheorem{theorem}{Theorem}[section]

\newtheorem{corollary}[theorem]{Corollary}
\newtheorem{lemma}[theorem]{Lemma}

\newtheorem{remark}[theorem]{Remark}
\newtheorem*{acknowledgment}{Acknowledgments}

\newcommand{\N}{\mathbb{N}}
\newcommand{\Q}{\mathbb{Q}}
\newcommand{\R}{\mathbb{R}}
\newcommand{\sfd}{{\sf d}}
\renewcommand{\d}{{\mathrm d}}
\newcommand{\ppi}{{\mbox{\boldmath$\pi$}}}
\newcommand{\ssigma}{{\mbox{\boldmath$\sigma$}}}
\newcommand{\mm}{\mathfrak m}
\newcommand{\entv}{{\rm Ent}_{\mm}}                 
\newcommand{\prob}[1]{\mathscr P(#1)}
\newcommand{\probt}[1]{\mathscr P_2(#1)}
\newcommand{\geo}{{\rm{Geo}}}                       
\newcommand{\e}{{\rm{e}}}
\newcommand{\Sp}{{\rm{Sp}}}                                                      
\newcommand{\gopt}{{\rm{OptGeo}}}                   
\newcommand{\res}[2]{{{\rm restr}_{#1}^{#2}}}
\newcommand{\restr}[1]{\lower3pt\hbox{$|_{#1}$}}
\newcommand{\diam}{\mathop{\rm diam}\nolimits}

\begin{document}

\title[Non-branching geodesics and optimal maps in strong $CD(K,\infty)$-spaces]
{Non-branching geodesics and optimal maps\\ in strong $CD(K,\infty)$-spaces}

\author{Tapio Rajala}
\address{Department of Mathematics and Statistics \\
         P.O. Box 35 (MaD) \\
         FI-40014 University of Jyv\"askyl\"a \\
         Finland}
\email{tapio.m.rajala@jyu.fi}

\author{Karl-Theodor Sturm}
\address{Institut f\"ur Angewandte Mathematik, Universit\"at Bonn\\
Endenicher Allee 60 \\
D-53115 Bonn \\
Germany }
\email{sturm@uni-bonn.de}

\date{\today}

\begin{abstract}
 We prove that in metric measure spaces where the entropy functional is $K$-convex along every Wasserstein geodesic
 any optimal transport between two absolutely continuous measures with finite second moments lives on
 a non-branching set of geodesics.
 As a corollary we obtain that in these spaces there exists only one optimal transport plan between any two
 absolutely continuous measures with finite second moments and this plan is given by a map.

 The results are applicable in metric measure spaces having Riemannian Ricci curvature bounded below,
 and in particular they hold also for Gromov-Hausdorff limits of Riemannian manifolds with Ricci curvature bounded from below
 by some constant.
\end{abstract}

\maketitle

\section{Introduction}

Ricci curvature lower bounds in general metric measure spaces were studied by the second author in \cite{S2006I, S2006II}
and at the same time with a similar approach by Lott and Villani in \cite{LV2009}. There the lower bound $K \in \R$ on the Ricci curvature without
a reference to the dimension of the metric measure space was defined as $K$-convexity of the entropy functional
along Wasserstein geodesics (see Section \ref{sec:preli} for details). These spaces are called (weak) $CD(K,\infty)$-spaces.
The word \emph{weak} is sometimes used to emphasize that the convexity is required only along one geodesic between any two given
probability measures.

A property of the $CD(K,\infty)$-spaces which complicates the theory is the possibility to have branching geodesics.
For example the space $\R^2$ with the $l_\infty$-norm is a $CD(0,\infty)$-space and it has lots of
branching geodesics. 
Being the limit as $p\to\infty$ (in any reasonable sense) of the spaces 
$\R^2$ with the $l_p$-norm -- which are non-branching 
CD$(K,\infty)$-spaces -- this example in particular illustrates that 
being non-branching is not a stable property.
A number of results in $CD(K,\infty)$-spaces have been proven only under the extra assumption that there are
no branching geodesics in the space. Although in some of the results this assumption has recently been removed (see for example
\cite{R2011, R2011b}), in many it still remains.

Because branching geodesics are hard to deal with, it is reasonable to consider more restrictive definitions that exclude
spaces with branching geodesics or at least limit the amount of branching that can occur.
One of the essential properties of the definitions of Ricci curvature lower bounds in metric measure spaces
is the stability under the measured Gromov-Hausdorff convergence. Therefore any stable definition that extends the Riemannian
case should include the limit spaces of Riemannian manifolds with uniform Ricci curvature lower bounds. While Riemannian
manifolds are known to be non-branching, to our knowledge it is not known whether this holds for their limit spaces. 
The spaces we consider in this paper cover also these limit spaces. Although our result does not rule out the possibility to 
have some branching geodesics, it still says that there has to be so few branching geodesics that optimal transports between 
any two absolutely continuous measures do not see them.

Before stating our main result we fix some terminology.
In this paper we will always assume $(X,\sfd)$ to be a complete separable geodesic metric space
and $\mm$ to be a locally finite Borel measure.
We refer to Section \ref{sec:preli} for some details on optimal mass transportation and $CD(K,\infty)$ condition in such spaces.
We call a space $(X,\sfd,\mm)$ \emph{essentially non-branching} if 
for every $\mu_0, \mu_1 \in \probt X$ which are absolutely continuous with respect to $\mm$
we have that any $\ppi \in \gopt(\mu_0,\mu_1)$ is concentrated on a set of non-branching geodesics.
The space $(X,\sfd, \mm)$ is said to be a strong $CD(K,\infty)$-space if the entropy 
$\entv$ is $K$-convex along every $\ppi \in \gopt(\mu_0,\mu_1)$ for every $\mu_0, \mu_1 \in \probt X$.

\begin{theorem}\label{thm:main}
Every strong $CD(K,\infty)$-space is essentially non-branching.
\end{theorem}

One of the cases which is covered by Theorem \ref{thm:main} are the $RCD(K,\infty)$-spaces, that is,
spaces with Riemannian Ricci curvature bounded from below by some constant $K \in \R$.
These spaces were recently defined and studied in \cite{AGS2011, AGS2011b}, see also \cite{AGMR2012}.
They were obtained by reinforcing the $CD(K,\infty)$-spaces with a requirement that the local structure of
the space must be Hilbertian.
The Hilbertian structure immediately rules out spaces like the above mentioned $\R^2$ with the $l_\infty$-norm.
In \cite{AGS2011b} it was shown that one of the equivalent formulations of $RCD(K,\infty)$-spaces is that
any probability measure with finite second moment is the starting point of an $\textrm{EVI}_K$-gradient flow of the entropy.
This condition is known to imply $K$-convexity of the entropy along every geodesic \cite{DS2008}.

In \cite{AGS2011b} it was also proven that the definition of $RCD(K,\infty)$ with a finite reference measure is 
stable under the measured Gromov-Hausdorff convergence (or under the $\mathbb D$-convergence introduced in \cite{S2006I}).
Later in \cite{GMS2012} the stability was proven with more general reference measures. 
When we combine the stability with Theorem \ref{thm:main} we arrive at the following corollary 
which in fact was our main motivation to write this paper.

\begin{corollary}\label{cor:rcd}
 The $RCD(K,\infty)$ condition  is stable under measured Gromov-Hausdorff convergence (or under the $\mathbb D$-convergence) and it implies essential non-branching.
\end{corollary}

As we already mentioned, we are not aware of any other non-branching results even for the
Gromov-Hausdorff limits of Riemannian manifolds with Ricci curvature bounded from below by some constant.

Another corollary of our result is that the $RCD(K,\infty)$ condition implies the formulation of $CD(K,\infty)$ that was
used by Lott and Villani \cite{LV2009}. They required convexity type inequalities for a class of functionals 
instead of just $\entv$, and hence their definition was at least a priori more restrictive than the definition of $CD(K,\infty)$
we use here, following \cite{S2006I}.
Corollary \ref{cor:LVCD} is proven exactly as in the non-branching $CD(K,\infty)$-spaces.
For the proof, see for instance \cite[Theorem 30.32]{V2009} or \cite[Proposition 4.2]{S2006II}.

\begin{corollary}\label{cor:LVCD}
 The $RCD(K,\infty)$ condition implies the $CD(K,\infty)$ condition by Lott and Villani.
\end{corollary}



By inspecting the proof of \cite[Theorem 3.3]{G2011} where the existence of optimal maps in non-branching $CD(K,\infty)$-spaces
was shown we can make a refinement of the statement of Theorem \ref{thm:main}.
Regarding this refinement, stated in Corollary \ref{cor:maps},  we note that in \cite{G2011} there was an extra assumption 
for the measures $\mu_0$ and $\mu_1$ to have finite entropy. This assumption was needed for showing that all absolutely
continuous measures with respect to a measure $\ppi \in \gopt(\mu_0,\mu_1)$ satisfying \eqref{eq:CDdef} also satisfy \eqref{eq:CDdef}.
Here this conclusion is already as an assumption so finiteness of the initial and final entropies are not needed.

\begin{corollary}\label{cor:maps}
 Let $(X,\sfd, \mm)$ be a strong $CD(K,\infty)$-space.
 Then for every $\mu_0, \mu_1 \in \probt X$ that are absolutely continuous with respect to $\mm$
 there is a unique $\ppi \in \gopt(\mu_0,\mu_1)$ and it is induced by a map.
\end{corollary}

Although Corollary \ref{cor:maps} follows quite easily from the proof of \cite[Theorem 3.3]{G2011} together with Theorem \ref{thm:main}
we will give at the end of the paper an outline of the proof together with some details that are different from the proof by Gigli.

In the classical situation of Euclidean spaces the existence of optimal transport maps was proven by Brenier in \cite{B1987}.
Since then there have been several generalizations of this result. Most relevant in the context of this paper are the
generalizations of McCann \cite{M2001} for Riemannian manifolds,
of Bertrand \cite{B2008} for Alexandrov spaces, and the most recent results of Gigli \cite{G2011} in non-branching $CD(K,N)$ and 
$CD(K,\infty)$-spaces
and of Ambrosio and the first author \cite{AR2011} in strongly non-branching metric spaces. Notice that as in the approach by
Gigli \cite{G2011} our proof for the existence of optimal transport maps does not use Kantorovich potentials.

\section{Preliminaries}\label{sec:preli}

Let us first recall some definitions and results related to optimal mass transportation and Ricci curvature lower bounds
using optimal mass transportation. More details on this subject can be found for example in the book by Villani \cite{V2009}.

In this paper we always work in a complete separable geodesic metric space $(X, \sfd)$ equipped with a locally finite Borel measure $\mm$. 

\subsection{Optimal mass transportation}
We write as $\prob X$ the set of all the probability measures defined on the $\sigma$-algebra consisting of universally measurable sets of $X$.
We denote by $\probt X$ the subset of $\prob X$ consisting of probability measures with finite second moments.
We equip the space $\probt X$ with the Wasserstein $2$-distance
which for any two measures $\mu_0, \mu_1 \in \probt X$ is defined as
\begin{equation}\label{eq:W2def}
  W_2(\mu_0, \mu_1) = \left(\inf\left\{\int_{X\times X} \sfd(x,y)^2 \,\d\sigma(x,y)\right\}  \right)^{1/2},
\end{equation}
where the infimum is taken over all $\sigma \in \probt{X \times X}$ with $\mu_0 = (\mathtt{p}^1)_\#\sigma$  
and $\mu_1 = (\mathtt{p}^2)_\#\sigma$. Here, and later on, $\mathtt{p}^k$ denotes the projection to the $k$-th coordinate.

We call a plan $\sigma \in \probt{X \times X}$ that minimizes \eqref{eq:W2def} an optimal plan. Optimal plans exist under were mild assumptions.
In contrast, the existence of optimal maps is rare. By an optimal map we mean Borel $T \colon X \to X$ 
for which the plan $G_\#\mu_0$ given by the graph $G(x) = (x,T(x))$ of $T$ is optimal. For the arguments in
this paper, it is crucial to notice that any subplan of an optimal plan is also optimal in the sense that for $\sigma$
optimal between $\mu_0$ and $\mu_1$ any $\tilde\sigma \ll \sigma$ is optimal between $(\mathtt{p}^1)_\#\tilde\sigma$
and $\mu_1 = (\mathtt{p}^2)_\#\tilde\sigma$.

Any optimal plan is concentrated on a cyclically monotone set $M \subset X \times X$. This means that for any
family $(x_1,y_1), \dots ,(x_n,y_n) \in M$ and any permutation $p \colon \{1, \dots n\} \to \{1, \dots n\}$ we have
\[
 \sum_{i=1}^n \sfd(x_i,y_i)^2 \le \sum_{i=1}^n \sfd(p(x_i),y_i)^2.
\]
Intuitively this just means that optimal plans can not be improved.

\subsection{Geodesics}
We denote by $C([0,1];X)$ the space of continuous curves $\gamma \colon [0,1] \to X$.
All the geodesics we consider in this paper are constant speed geodesics parametrized by the unit interval $[0,1]$.
We denote the set of all such geodesics in $X$ by $\geo(X) \subset C([0,1];X)$, and for $\gamma \in C([0,1];X)$ and $t \in [0,1]$
we use the abbreviation $\gamma_t = \gamma(t)$.
The distance between $\gamma^1,\gamma^2 \in C([0,1];X)$ is given by $\sfd^*(\gamma^1,\gamma^2) = \max\{\sfd(\gamma_t^1,\gamma_t^2)\,:\,0\le t\le 1\}$.
Notice that the subspace $(\geo(X),\sfd^*)$, with $\sfd^*$ restricted to $\geo(X) \times \geo(X)$,
is complete and separable because the space $(X,\sfd)$ is. (This is because
$\geo(X)$ is a $\sfd^*$-closed subset of the separable space $C([0,1];X)$.)
We write $l(\gamma) = \sfd(\gamma_0,\gamma_1)$ for any geodesic $\gamma \in \geo(X)$ .

We define for all $s,t \in [0,1]$ the restriction map $\res{s}{t} \colon C([0,1];X) \to C([0,1];X) \colon \gamma \mapsto \gamma \circ f_s^t$
with $f_s^t \colon [0,1] \to [0,1] \colon x \mapsto (t-s)x + s$. Notice that $\res{s}{t}(\geo(X)) \subset \geo(X)$.
We call a set $\Gamma \subset \geo(X)$ \emph{non-branching} if for any $\gamma, \tilde\gamma \in \Gamma$ we have:
if there exists $t \in (0,1)$ such that $\res{0}{t}\gamma = \res{0}{t}\tilde\gamma$, then $\gamma = \tilde\gamma$.
A measure $\ppi\in \prob{\geo(X)}$ is said to be concentrated on a set of non-branching geodesics, if there exists
a non-branching Borel set $\Gamma \subset \geo(X)$ so that $\ppi(\Gamma)=1$.
The space consisting of all measures $\ppi\in \prob{\geo(X)}$
for which the mapping $t \mapsto (\e_t)_\#\ppi$ is a geodesic in $\prob X$ from $\mu_0 = (\e_0)_\#\ppi$ to $\mu_1 = (\e_1)_\#\ppi$ 
is denoted by $\gopt(\mu_0, \mu_1)$. Here the evaluation map is defined as $\e_t \colon C([0,1];X) \to X \colon \gamma \mapsto \gamma_t$.

For every geodesic $(\mu_t) \in \geo(\probt{X})$ joining $\mu_0$ to $\mu_1$, there also exists a measure
$\ppi \in \gopt(\mu_0, \mu_1)$ with $(e_t)_\#\ppi = \mu_t$. See for instance \cite[Theorem 2.10]{AG2013} for the proof of this fact. Notice that the measure $\ppi$ is not necessarily unique.

Our proof is heavily based on restricting a given $\ppi \in \gopt(\mu_0,\mu_1)$. Without mentioning
it every time, we use the fact that for a Borel $f \colon \geo(X) \to [0,\infty)$ with $\int_{\geo(X)} f\,\d\ppi = 1$
we have $f\ppi \in \gopt((e_0)_\#f\ppi,(e_1)_\#f\ppi)$, analogously to the restrictions of optimal plans. 
Another fact which we will repeatedly use is that $(\res{t}{s})_\#\ppi \in \gopt((e_t)_\#\ppi,(e_s)_\#\ppi)$.

\subsection{Ricci curvature lower bounds}
The Ricci curvature lower bounds are defined using the entropy functional $\entv \colon \prob X \to [-\infty,\infty]$. 
It is defined as
\[
 \entv(\mu) = \int_X \rho \log \rho \,\d\mm
\]
for any absolutely continuous measure $\mu = \rho\mm \in \prob X$ for which the positive part of
$\rho \log \rho$ is integrable.
For other measures in $\prob X$ 
we define $\entv(\mu) = \infty$. Notice that for an absolutely continuous measure with support in a set of finite $\mm$-measure 
the negative part of $\rho \log \rho$ is always integrable. As we will later see we may always assume the measures
to be supported in a set of finite $\mm$-measure in the proof of Theorem \ref{thm:main}.

Following the definition in \cite{S2006I} we call a metric measure space $(X,\sfd,\mm)$
a (weak) $CD(K,\infty)$-space, with some $K \in \R$, provided that for any $\mu_0, \mu_1 \in \probt X$
that are absolutely continuous with respect to $\mm$ 
there exists a geodesic $(\mu_t) \in \geo(\probt X)$ along which $\entv$ is $K$-convex, that is
\begin{equation}\label{eq:CDdef}
 \entv(\mu_t) \le (1-t)\entv(\mu_0) + t \entv(\mu_1)- \frac{K}{2}t(1-t)W_2^2(\mu_0,\mu_1)
\end{equation}
holds for all $t \in [0,1]$. We say that a functional is $K$-convex along $\ppi \in \gopt(\mu_0,\mu_1)$
if it is $K$-convex along the corresponding geodesic $(\e_t)_\#\ppi$.
If the inequality \eqref{eq:CDdef} is required to hold for all geodesics
 $(\mu_t) \in \geo(\probt X)$ with endpoints absolutely continuous with respect to $\mm$, the space is 
called a strong $CD(K,\infty)$-space. 

In the proofs we will need $\entv$ to be convex along restrictions $f\ppi$.
If the space $(X, \sfd, \mm)$ would only be a weak $CD(K,\infty)$-space, $K$-convexity along restrictions would not be guaranteed.
However, in strong $CD(K,\infty)$-spaces it follows directly from the definition.

\section{Proof of Theorem \ref{thm:main}}

The idea of the proof is similar to the proof of \cite[Theorem 4]{R2011}. We prove the claim by contradiction.
First we find two geodesics in the Wasserstein space
which start as the same geodesic and then branch out to two completely disjoint ones.
By comparing the two geodesics separately and on the other hand their sum as one geodesic,
we notice that there will be an extra drop of $\log 2$ in the entropy during the time interval
when the branching occurs. In order to arrive at a contradiction
this branching has to happen in a small enough time interval as in the Figure \ref{fig:contrad}.

There are two difficult steps in the proof before we arrive at the contradiction mentioned above.
First of all we have to find the two geodesics that branch out to two
completely separate ones. Once we have found them, we have to restrict the measures so that branching happens in a
small enough interval. Here the difficulty is that when we restrict the measures, also their marginals change.
By choosing the correct restrictions we can overcome this problem.

\begin{figure}
  \psfrag{0}{$0$}
  \psfrag{1}{$1$}
  \psfrag{t1}{$t_1$}
  \psfrag{t2}{$t_2$}
  \psfrag{g1}{$\Gamma_1$}
  \psfrag{g2}{$\Gamma_2$}
  \psfrag{ent}{$\entv$}
  \psfrag{l}{$\log 2$}
  \centering
   \includegraphics[width=0.5\textwidth]{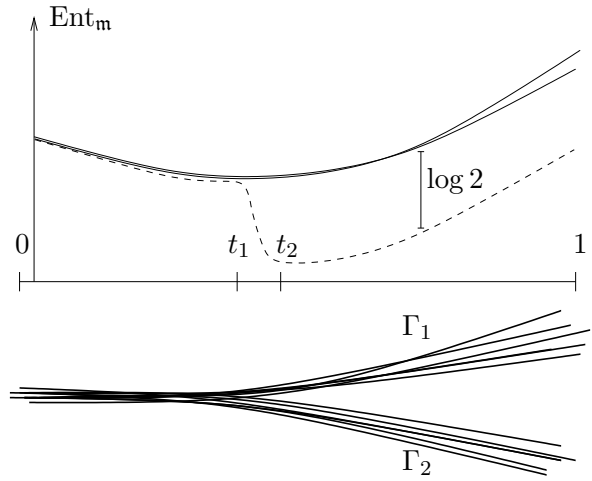}
  \caption{The idea in the proof of Theorem \ref{thm:main} is to find two sets of geodesics $\Gamma_1$ and $\Gamma_2$ so that 
           the transport supported on the union of the geodesics is a measure $\ppi$ satisfying the assumption 
           of Theorem \ref{thm:main} and so that the two sets of geodesics
           agree until time $t_1$ and then branch out so that they become totally separated after time $t_2$.
           The entropies along the set $\Gamma_1$ and along the set $\Gamma_2$ are illustrated by the solid $K$-convex graphs.
           The entropy along $\Gamma_1 \cup \Gamma_2$ is illustrated by the dashed graph. The non-$K$-convexity of this
           graph due to the drop of $\log 2$ in the entropy contradicts the assumption of the theorem.}
  \label{fig:contrad}
 \end{figure}

Let us start the proof with a simple lemma. It allows us to select the two disjoint geodesics
needed for the contradiction.
The geodesics will be selected using a probability measure on the product space $\geo(X)\times \geo(X)$.
Notice that the constant $1/5$ in the lemma is not sharp, however it is sufficient for our use.

\begin{lemma}\label{lma:largedecomp}
 Let $(X,\sfd)$ be a separable metric space.
 Then for any $\ssigma \in \prob{X \times X}$ for which $\ssigma(\{(x,x)\,:\,x\in X\}) = 0$
 there exists $E \subset X$ so that $\ssigma(E \times (X\setminus E)) > 1/5$.
\end{lemma}
\begin{proof}
 We will first reduce the general case to the case where $X$ is a finite set. Take $\epsilon > 0$.
 Because the diagonal $\{(x,x)\,:\,x\in X\}$ has zero $\ssigma$-measure, there exists $\delta > 0$ so that
 \begin{equation}\label{eq:deltaneigh}
  \ssigma\big(\big\{(x,y) \in X \times X \,:\, \sfd(x,y)<\delta\big\}\big) < \epsilon.  
 \end{equation}
 Partition the space $X$ into a countable collection of Borel sets $(Q_i)_i$ with diameter
 at most $\delta/2$. Now there exists some $n \in \N$ so that
 \begin{equation}\label{eq:ncubes}
  \sum_{i=1}^n\left((\mathtt{p}^k)_\#\ssigma\right)(Q_i) > 1-\epsilon, \quad \text{for }k = 1,2. 
 \end{equation}
 Therefore, by combining \eqref{eq:deltaneigh} and \eqref{eq:ncubes} we have
 \[
   \ssigma\left(\bigcup_{1\le i,j \le n, i \ne j} Q_i\times Q_j\right) \ge 1 - 3\epsilon
 \]
 and so by forgetting a part of the space with arbitrarily small measure we may assume the space $X$ to consist of $n$ points.
 
 The existence of the set $E$ in the case where $X$ consists of $n$ points follows easily:
 There are a total number of $2^n-2$ ways to select a non-empty set $E \subsetneq X$ and
 for any pair of points $(i,j)$ with $i \ne j$ there are $2^{n-2}$ sets $E \subset X$ with $i \in E$ and $j \notin E$.
 As a consequence
 \[
  \sum_{E \subset X} \ssigma\big(E \times (X \setminus E)\big) = 2^{n-2}
 \]
 and so there has to be a set $E \subset X$ with
 \[
  \ssigma\big(E \times (X \setminus E)\big) \ge \frac{2^{n-2}}{2^n-2} > \frac14.
 \]
 Taking $\epsilon >0$ sufficiently small finishes the proof.
\end{proof}

Now we are ready to continue with the proof of Theorem \ref{thm:main}. The first reduction steps contain
ideas which were used in the proofs of \cite[Theorem 3.3]{G2011} and \cite[Theorem 4]{R2011}. The rest of the proof is then
close to that of \cite[Theorem 4]{R2011}.

\begin{proof}[Proof of Theorem \ref{thm:main}]
 Let us give the idea behind each step of the proof.
 We will prove the claim by contradiction, so we have a measure $\ppi$ that does not live on a non-branching
 set of geodesics.
 The first step is to restrict the measures in time and space to live inside a sufficiently small ball. 
 This will help in the last step of the proof in estimates involving the extra term coming from $K$-convexity when $K \ne 0$.
 In the second step we produce via disintegration a measure $\ssigma$ on the product space $\geo(X) \times \geo(X)$
 which gives positive measure to pairs of geodesics which start as the same but after some time $T$ branch out.

 In the third step we pushforward the measure $\ssigma$ on the product space to live on more suitable pairs of geodesics. The
 pairs of geodesics which we want to avoid are the ones that branch out and come back together infinitely often.
 This step will be needed only to achieve the reductions in the next step.
 In the fourth step of the proof we restrict the measure to geodesics which stay disjoint at least for some time $\delta>0$
 (independent of the geodesics) immediately after the branch out for the first time.
 We also restrict the measure so that its marginals have bounded densities.
 
 In the fifth step we find, for a given $\epsilon>0$, a time $t$ so that during the time interval $[t,t+\epsilon]$
 the measure sees lots of branching. In the sixth step we first use Lemma \ref{lma:largedecomp} to obtain two disjoint
 sets of geodesics so that their product has large measure. After this we further restrict the measure so that we can project
 the product measure to two measures $\ppi^{\text{up}}$ and $\ppi^{\text{down}}$ which have disjoint supports at time $t+\epsilon$.
 Finally in the last step, step seven we compare the entropies along $\ppi^{\text{up}}$, $\ppi^{\text{down}}$ and $\ppi^{\text{up}} + \ppi^{\text{down}}$
 to obtain a contradiction which we already mentioned in the beginning of this section and in the Figure \ref{fig:contrad}.

 \textbf{Step 1: Localization to a small ball.}\\
 Assume that Theorem \ref{thm:main} is not true. Then there exist $\mu_0, \mu_1 \in \probt X$ and an optimal 
 $\ppi \in \gopt(\mu_0,\mu_1)$ that is not concentrated on non-branching geodesics. Because the space $(X,\sfd)$
 is separable and the measure $\mm$ locally finite, we can cover $(X,\sfd)$ with a countable collection of balls $B(x_i,l_i/4)$
 so that 
 \begin{equation}\label{eq:lengthbound}
   l_i \le \sqrt{\frac{\log 2}{6|K|+1}}
 \end{equation} 
 and $\mm(B(x_i,l_i)) < \infty$. Since $\ppi$ is not concentrated on non-branching geodesics there exist some $i \in \N$ and $L >l_i$ so that
 $\ppi$ has some branching inside $B(x_i,l_i/4)$ along geodesics with length at most $L$.
 That is, $\ppi(\Gamma)<1$ for every $\Gamma \subset \geo(X)$ satisfying the following: if
 $\gamma^1, \gamma^2 \in \Gamma$ so that $\gamma^1_s = \gamma^2_s$ for all $s \in [0,t_0]$ with $\gamma^1_{t_0} \in B(x_i,l_i/4)$
 and $l(\gamma^1) \le L$, then $\gamma^1_s = \gamma^2_s$ also for all $s$ for which $\gamma^1_s \in  B(x_i,l_i/4)$ or $\gamma^2_s \in  B(x_i,l_i/4)$.

 Therefore there exists $t_1 \in (0,1-l_i/(4L))$ so that
 \[
  \ppi^r = (\res{t_1}{l_i/(4L)})_\#\left(\frac{1}{\ppi(\Gamma^r)}\ppi|_{\Gamma^r}\right),  
 \]
 with $\ppi(\Gamma^r)>0$ defined as
 \[
  \Gamma^r = \left\{\gamma \in \geo(X) \,:\,l(\gamma) \le L \text{ and }\gamma_{t_1 + l_i/(8L)} \in B(x_i,l_i/4) \right\},
 \]
 is not concentrated on non-branching geodesics. Now the measure $\ppi^r$ is supported on a set of geodesics that live
 inside the ball $B(x_i,l_i/2)$. Thus without loss of generality, we may assume from the beginning that the original measure $\ppi$ is concentrated
 on geodesics living inside some ball $B(x,l/2) \subset X$ with $l$ having the same bound from above as $l_i$ in \eqref{eq:lengthbound} and $\mm(B(x,l))< \infty$.

 \textbf{Step 2: Disintegrating the branching measure.}\\
 We claim that from the assumption that $\ppi$ is not concentrated on a non-branching set of geodesics we know that there exists some 
 $T \in (0,1)$ so that the measure $\ppi_\gamma$ is not a Dirac mass for a $(\res{0}{T})_\#\ppi$-positive set of
 curves $\gamma \in \geo(X)$, where $\{\ppi_\gamma\} \subset \prob{\geo(X)}$ is the disintegration of $\ppi$ with respect to $\res{0}{T}$. 
 The measure $\ppi_\gamma$ not being a Dirac mass means that it is concentrated on geodesics which coincide with $\gamma$ in $[0,T]$ and that a set
 of $\ppi_\gamma$-positive measure of geodesics do branch after time $T$.
 We prove the claim by contradiction. Assume that it is not true so that for any $T \in (0,1)$ the measures
 $\ppi_\gamma$ are Dirac for $(\res{0}{T})_\#\ppi$-almost every curve $\gamma \in \geo(X)$. Take $\epsilon > 0$. Since $\geo(X)$
 is separable there exists a countable Borel decomposition of $\geo(X)$ into disjointed sets $(A_{i,\epsilon})_{i\in I}$ with 
 $\diam(A_{i,\epsilon})<\epsilon$. For each $i \in I$ the map $M_{i,\epsilon, T} \colon \geo(X) \to [0,1] \colon \gamma \mapsto \ppi_\gamma(A_{i,\epsilon})$
 is Borel measurable and so $M_{i,\epsilon, T}^{-1}(\{1\})$ is a Borel set. Since $\ppi_\gamma$ are probability measures,
 $(M_{i,\epsilon, T}^{-1}(\{1\}))_{i \in I}$ are disjointed.
 By the assumption that almost every $\ppi_\gamma$ is a Dirac mass the Borel set
 \[
  \Gamma_{\epsilon, T} = \bigcup_{i \in I}\left((\res{0}{T})^{-1}(M_{i,\epsilon, T}^{-1}(\{1\})) \cap A_{i,\epsilon} \right)
 \]
 has full $\ppi$-measure. Therefore also the Borel set
 \[
  \Gamma_T = \bigcap_{\epsilon \in (0,1) \cap \Q}\Gamma_{\epsilon, T}
 \]
 has full $\ppi$-measure. Notice that for $\gamma^1, \gamma^2 \in \Gamma_T$ with $\gamma^1 \ne \gamma^2$ there exists
 $\epsilon \in (0,1)\cap \Q$ so that $\gamma^1 \in A_{i,\epsilon}$ and $\gamma^2 \in A_{j,\epsilon}$ with $i \ne j$.
 Therefore $\res{0}{T}(\gamma^1) \in M_{i,\epsilon, T}^{-1}(\{1\})$ and $\res{0}{T}(\gamma^2) \in M_{j,\epsilon, T}^{-1}(\{1\})$
 giving $\res{0}{T}(\gamma^1) \ne \res{0}{T}(\gamma^2)$.
 This means that the set $\Gamma_T$ does not contain geodesics that branch after time $T$.
 Thus the Borel set 
 \[
  \bigcap_{T \in (0,1) \cap \Q}\Gamma_{T}
 \]
 is non-branching and has full $\ppi$-measure, giving a contradiction.

 Consider a measure $\ssigma \in \prob{\geo(X) \times \geo(X)}$ obtained by integrating up the disintegrated measure $\ppi_\gamma$
 as a product measure $\ppi_\gamma\times\ppi_\gamma$. That is, for any Borel measurable $f \colon \geo(X)\times\geo(X) \to [0,\infty]$ we have
 \begin{align*}
  &\int_{\geo(X) \times \geo(X)}f(\gamma^1,\gamma^2)\,\d\ssigma(\gamma^1,\gamma^2) \\
 &\quad = \int_{\geo(X)}\int_{(\res{0}{T})^{-1}(\gamma) \times (\res{0}{T})^{-1}(\gamma)}f(\gamma^1,\gamma^2)\,\d(\ppi_\gamma\times\ppi_\gamma)(\gamma^1,\gamma^2)\,\d(\res{0}{T})_\#\ppi(\gamma).
 \end{align*}
 The measures $\ppi_\gamma\times\ppi_\gamma$ are illustrated in Figure \ref{fig:product}.
\begin{figure}
  \psfrag{t}{$t$}
  \psfrag{x}{$X$}
  \psfrag{T}{$T$}
  \psfrag{p}{$\ppi_\gamma$}
  \psfrag{pp}{$\ppi_\gamma\times\ppi_\gamma$}
  \centering
  \includegraphics[width=0.8\textwidth]{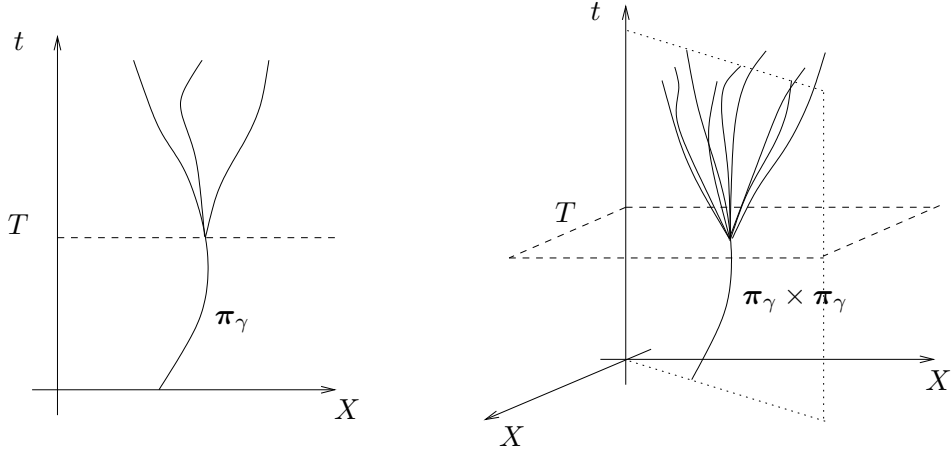}
  \caption{The measure $\ssigma$ is constructed by integrating up the measures $\ppi_\gamma\times\ppi_\gamma$.
    Since branching for $\ppi_\gamma$ occurs only after time $T$, all the measures $\ppi_\gamma\times\ppi_\gamma$ and hence also the measure $\ssigma$
    live on the diagonal of $X \times X$ until time $T$.}
  \label{fig:product}
 \end{figure}

 \textbf{Step 3: Reducing the number of branching points of pairs of geodesics.}\\
 It might happen that part of the measure $\ssigma$ lives on pairs of geodesics $(\gamma^1,\gamma^2)$ for which the set 
 $\{t \in [0,1] \,:\, \gamma_t^1 = \gamma_t^2\}$ consists of infinitely many intervals. We want to exclude this behaviour
 so that branching geodesics stay disjoint for a positive time immediately after branching. We will do this by a measurable
 selection. 

 We start by defining the subsets $\mathfrak{D}, \mathfrak{B}, \mathfrak{B}^1$ of $\geo(X)^2$. The first one is the diagonal set
 \[
  \mathfrak{D} := \{(\gamma^1,\gamma^2) \in \geo(X)^2 \,:\, \gamma^1 = \gamma^2\},
 \]
 the second one the set of pairs of geodesics branching after time $T$
 \[
  \mathfrak{B} := \{(\gamma^1,\gamma^2) \in \geo(X)^2\setminus \mathfrak{D} \,:\, \res{0}{T}\gamma^1 = \res{0}{T}\gamma^2\}
 \]
 and the final one the set of pairs of geodesics that branch exactly once after time $T$
 \[
  \mathfrak{B}^{1} := \left\{(\gamma^1,\gamma^2) \in \mathfrak{B} \,:\, \text{the set } \{t \in [0,1] \,:\, \gamma_t^1 \ne \gamma_t^2\} 
 \text{ is an interval}\right\}.
 \]
 The subsets $\mathfrak{D}$ and $\mathfrak{B}$ are clearly Borel. Let us show that $\mathfrak{B}^{1}$ is also Borel.
 For every $\epsilon, t_1, t_2 > 0$ define the closed sets
 \begin{align*}
  \mathfrak{B}_{\epsilon,t_1,t_2} = \{(\gamma^1,\gamma^2) \in \mathfrak{B} \,:\, & \text{ there exist } s \in [t_1,t_2]\text{ with } \gamma_{s}^1=\gamma_{s}^2 \text{ and}\\
        & r_1 \in [T,t_1], r_2 \in [t_2,1] \text{ such that }\sfd(\gamma_{r_i}^1,\gamma_{r_i}^2) \ge \epsilon \text{ for }i =1,2\}.
 \end{align*}
 Now
 \[
  \mathfrak{B}^1 = \mathfrak{B} \setminus \left(\bigcup_{\epsilon,t_1,t_2 \in (0,1]\cap \Q}\mathfrak{B}_{\epsilon,t_1,t_2} \right)
 \]
 and so $\mathfrak{B}^1$ is Borel.
 
 Now define $\mathtt{End} \colon \mathfrak{B}^{1} \to X^3 \colon (\gamma^1,\gamma^2) \mapsto (\gamma_0^1,\gamma_1^1,\gamma_1^2)$ and
 let $\mathtt{End}^{-1} \colon \mathtt{End}(\mathfrak{B}^1) \to \mathfrak{B}^1$ be its Suslin measurable right-inverse given by the Jankoff theorem \cite{J1941}. 
 (The set $\mathfrak{B}^{1}$ is a Suslin space as a Borel subset of a Polish space.
  The mapping $\mathtt{End}$ is continuous and thus $\mathtt{End}(\mathfrak{B}^1)$ is also Suslin space.
   Suslin subsets of a Polish space are universally measurable and so Suslin-measurability suffices for our considerations.)

 Now consider the Suslin measurable map $\mathtt{Br} \colon \geo(X)^2 \to \geo(X)^2$ given by
 \[
  \mathtt{Br}(\gamma^1,\gamma^2) = \begin{cases}
                        \mathtt{End}^{-1}(\gamma_0^1,\gamma_1^1,\gamma_1^2) & \text{if }(\gamma_0^1,\gamma_1^1,\gamma_1^2) \in \mathtt{End}(\mathfrak{B}^{1}),\\
                        (\gamma^1,\gamma^2) & \text{otherwise}.
                       \end{cases}
 \]
 The role of $\mathtt{Br}$ is to select an element from $\mathfrak{B}^1$ with the same starting point $\gamma_0^1$ and endpoints $\gamma_1^1$ and $\gamma_1^2$,
 if there exists one, and otherwise to give the original pair. For any branching pair of geodesics it gives a pair of geodesics
 that branches exactly once, like in Figure \ref{fig:1branch}.
 (Such geodesics exist:
 If $\gamma_1^1 \ne \gamma_1^2$, let $s \in (0,1)$ be the last time when $\gamma_s^1 = \gamma_s^2$. Then defining
 $\gamma^3|_{[0,s]} = \gamma^4|_{[0,s]} = \gamma^1|_{[0,s]}$, $\gamma^3|_{[s,1]} = \gamma^1|_{[s,1]}$ and 
 $\gamma^4|_{[s,1]} = \gamma^2|_{[s,1]}$ gives you the desired pair of geodesics $(\gamma^3, \gamma^4)$.
 If $\gamma_1^1 = \gamma_1^2$, take $0 < s_1 < s_2 \le 1$ such that $\gamma_{s_1}^1 = \gamma_{s_2}^2$, 
 $\gamma_{s_2}^1 = \gamma_{s_2}^2$ and $\gamma_{s}^1 \ne \gamma_{s}^2$ for all $s \in (s_1,s_2)$.
 Then define $\gamma^3|_{[0,s_1]\cup[s_2,1]} = \gamma^4|_{[[0,s_1]\cup[s_2,1]} = \gamma^1|_{[0,s_1]\cup[s_2,1]}$,
 $\gamma^3|_{[s_1,s_2]} = \gamma^1|_{[s_1,s_2]}$ and 
 $\gamma^4|_{[s_1,s_2]} = \gamma^2|_{[s_1,s_2]}$, and again you have the desired pair of geodesics $(\gamma^3, \gamma^4)$.)
 In other words $\mathtt{Br}(\mathfrak{B}) \subset \mathfrak{B}^1$.
 Now the pushforward measure $\mathtt{Br}_\#\ssigma$ gives $\mathfrak{B}^1$ positive measure and it has the same marginals as $\ssigma$, in other words,
 $(e_0,e_0)_\#(\mathtt{Br}_\#\ssigma) = (e_0,e_0)_\#\ssigma$ and $(e_1,e_1)_\#(\mathtt{Br}_\#\ssigma) = (e_1,e_1)_\#\ssigma$.
 Therefore, after pushing forward the measure by $\mathtt{Br}$ if necessary, we may assume that $\ssigma(\mathfrak{B}^1) > 0$.
\begin{figure}
  \psfrag{g1}{$\gamma^1$}
  \psfrag{g2}{$\gamma^2$}
  \psfrag{g3}{$\gamma^3$}
  \psfrag{g4}{$\gamma^4$}
  \psfrag{Br}{$\mathtt{Br}$}
  \centering
  \includegraphics[width=0.9\textwidth]{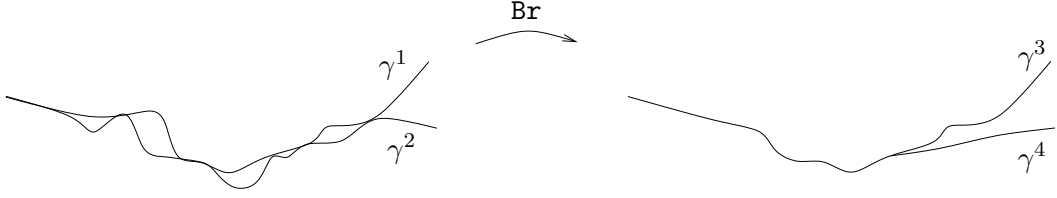}
  \caption{The mapping $\mathtt{Br}$ takes a branching pair of geodesics $(\gamma^1,\gamma^2)$ to another pair of geodesics $(\gamma^3,\gamma^4)$ that have the same endpoints, but branch only once.}
  \label{fig:1branch}
 \end{figure}

 \textbf{Step 4: Further restrictions of the measure $\ssigma$.}\\
 Next we will show that by restricting and rescaling the measure $\ssigma$ we may assume that there exist $S \in (T,1)$,
 $\delta \in (0,(1-S)/2)$ and $C > 0$ so that
\begin{equation}\label{eq:boundeddensity}
    \frac{\d(\e_0)\#(\mathtt{p}^1\ssigma)}{\d\mm},\frac{\d(\e_1)\#(\mathtt{p}^1\ssigma)}{\d\mm}, \frac{\d(\e_1)\#(\mathtt{p}^2\ssigma)}{\d\mm} < C  
 \end{equation}
 and for $\ssigma$-almost every $(\gamma^1,\gamma^2)$ there exists $t \in [T,S]$ such that
 \begin{equation}\label{eq:deltaseparation}
  \res{0}{t}(\gamma^1) = \res{0}{t}(\gamma^2) \text{ and } \gamma_s^1 \ne \gamma_s^2 \text{ for all } s\in(t,t+\delta).
 \end{equation}
 
 First of all 
 we may select $S \in (T,1)$ such that
 \[
   \ssigma\big(\big\{(\gamma^1,\gamma^2) \in \mathfrak{B}^1\,:\,\res{0}{S}(\gamma^1) \ne \res{0}{S}(\gamma^2)\big\}\big) > 0.
 \]
 On the other hand the set
 \[
  \big\{(\gamma^1,\gamma^2) \in \mathfrak{B}^1\,:\,\res{0}{S}(\gamma^1) \ne \res{0}{S}(\gamma^2)\big\}
 \]
 can be covered by a countable union of sets of the form
 \begin{align*}
 A_q = \big\{(\gamma^1,\gamma^2) \in \mathfrak{B}^1\,:\, \text{there exists } t \in [T,S] & \text{ such that }\res{0}{t}(\gamma^1) = \res{0}{t}(\gamma^2)\\
                                           & \text{ and } \gamma_s^1 \ne \gamma_s^2 \text{ for all } s\in(t,t+q)\big\}.
 \end{align*}
 Therefore there exists a $\delta \in (0,(1-S)/2)$ so that
 \[
 \ssigma(A_\delta) > 0.
 \]
 Notice that the sets $A_q$ are Borel because they can be written as
 \[
  A_q = \bigcap_{z \in (q,1)\cap\Q}\bigcup_{t,\epsilon \in (0,1-z]\cap \Q}\left\{(\gamma^1,\gamma^2) \in \mathfrak{B}^1\,:\, \sfd(\gamma_s^1,\gamma_s^2) \ge \epsilon \text{ for all } s\in [t,t+z]\right\}.
 \]
 By restricting $\ssigma$ to $A_\delta$ we have then obtained \eqref{eq:deltaseparation}. Since $\mu_0$ and $\mu_1$ are absolutely continuous with respect
 to $\mm$ a further restriction and finally rescaling yields \eqref{eq:boundeddensity}.

 \textbf{Step 5: Selecting a time $t$ where sufficient amount of branching occurs.}\\
 Define the function $f \colon [0,1] \to [0,1]$ as
 \[
  f(t) = \ssigma\big(\big\{(\gamma^1,\gamma^2) \,:\, \res{0}{t}\gamma^1 = \res{0}{t}\gamma^2\big\}\big).
 \]
 Notice that $f$ is decreasing, $f = 1$ on $[0,T]$ and $f = 0$ on $[S,1]$. Roughly speaking $f(t)$
 gives the amount of branching that will occur after time $t$.

 Take a small $\epsilon \in (0,\delta)$ and let $t \in [T,S]$ be such that $f(t) - f(t+\epsilon) > \epsilon$.
 Such $t$ exists because $f$ decreases by $1$ on the interval $[T,S]$.
 Define
 \[
  F = \big\{(\gamma^1,\gamma^2) \,:\, \res{0}{t}\gamma^1 = \res{0}{t}\gamma^2 \text{ and }\res{0}{t+\epsilon}\gamma^1 \ne \res{0}{t+\epsilon}\gamma^2\big\}.
 \]
 Now $\ssigma(F) > \epsilon$ by the choice of $t$ and \eqref{eq:deltaseparation}. 

 \textbf{Step 6: Finding two measures completely branching out during $[t,t+\epsilon]$.}\\
 By Lemma \ref{lma:largedecomp} there exists $E \subset \geo(X)$ so that
 \[
  \ssigma\Big(\big(E\times (\geo(X)\setminus E)\big)\cap F\Big) > \frac{\ssigma(F)}{5} > \frac{\epsilon}{5}.
 \]
 Let $\eta > 0$ be such that 
 \begin{equation}\label{eq:mostareapart}
  M := \ssigma(A) > \frac{\ssigma\Big(\big(E\times (\geo(X)\setminus E)\big)\cap F\Big)}2 > \frac\epsilon{10},
 \end{equation}
 where $A = \big\{(\gamma^1,\gamma^2) \in \big(E\times (\geo(X)\setminus E)\big)\cap F \,:\, \sfd(\gamma_{t+\epsilon}^1,\gamma_{t+\epsilon}^2)> \eta\big\}$.
 Such $\eta$ exists since $\gamma_{t+\epsilon}^1 \ne \gamma_{t+\epsilon}^2$ for $\ssigma$-almost every $(\gamma^1,\gamma^2)$ by 
 \eqref{eq:deltaseparation} and the fact that $\epsilon < \delta$.

 Now define $\tilde\ssigma = \ssigma\restr{A}$ and using it $\tilde\ppi = (\mathtt{p}^1)_\#\tilde\ssigma$.
 Write $\tilde\rho_s\mm  = (\e_s)_\#\tilde\ppi$ for all $s \in [0,1]$.
 By Jensen's inequality and \eqref{eq:mostareapart} we get
 \begin{equation}\label{eq:lowerbound}
  \int\tilde\rho_{t+\epsilon} \log \tilde\rho_{t+\epsilon}\,\d\mm \ge M \log \frac{M}{\mm(B(x,l/2))} \ge M \log \frac{\epsilon}{10\mm\big(B(x,l/2)\big)}.
 \end{equation}

 Since $X$ is separable we can decompose it into a countable collection $\{Q_i\}$ of disjointed Borel
 sets with diameter less than $\eta$.
 We claim that for at least one $Q \in \{Q_i\}$ we have
 \[
  w := \int_{Q}\tilde\rho_{t+\epsilon}\,\d\mm = \tilde\ppi(G) > 0
 \]
 for $G = \{\gamma \in \geo(X) \,:\, \gamma_{t + \epsilon} \in Q \}$ and
 \begin{equation}\label{eq:lowerboundsmall}
  \int\rho_{t+\epsilon}^{\text{up}} \log \rho_{t+\epsilon}^{\text{up}}\,\d\mm \ge w \log \frac{\epsilon}{10\mm\big(B(x,l/2)\big)},
 \end{equation} 
 where $\ppi^{\text{up}} = \tilde\ppi\restr{G} =  (\mathtt{p}^1)_\#\tilde\ssigma\restr{G\times \geo(X)}$ and
 $\mu_s^\text{up} = \rho_s^\text{up}\mm  = (\e_s)_\#\ppi^{\text{up}}$ for all $s \in [0,1]$.
 For if this were not the case we would have
 \begin{align*}
  \int\tilde\rho_{t+\epsilon} \log \tilde\rho_{t+\epsilon}\,\d\mm = \sum_{i} \int_{Q_i}\tilde\rho_{t+\epsilon} \log \tilde\rho_{t+\epsilon}\,\d\mm
  & < \sum_{i} \int_{Q_i}\tilde\rho_{t+\epsilon} \,\d\mm \log \frac{\epsilon}{10\mm\big(B(x,l/2)\big)}\\
  & = M \log \frac{\epsilon}{10\mm\big(B(x,l/2)\big)}
 \end{align*}
 contradicting \eqref{eq:lowerbound}.

 Define also $\ppi^{\text{down}} = (\mathtt{p}^2)_\#\tilde\ssigma\restr{G\times \geo(X)}$
 and $\mu_s^\text{down} = \rho_s^\text{down}\mm  = (\e_s)_\#\ppi^{\text{down}}$ for all $s \in [0,1]$.
 Notice that $(\res{0}{t})_\#\ppi^{\text{up}} = (\res{0}{t})_\#\ppi^{\text{down}}$
 and $\mu_{t+\epsilon}^\text{up} \perp \mu_{t+\epsilon}^\text{down}$.
 
 By \eqref{eq:boundeddensity} the estimate
 \begin{equation}\label{eq:upperbounds}
  \int \rho \log \rho \,\d\mm\le w \log C
 \end{equation}
 holds for all $\rho \in \{\rho_0^{\text{down}} = \rho_0^{\text{up}}, \rho_1^{\text{down}}, \rho_1^{\text{up}}\}$.

 \textbf{Step 7: Comparing entropies along different geodesics.}\\
 Now, similarly as in \cite{R2011} we can estimate using the $K$-convexity
 first along the measure $(\ppi^{\text{up}}+\ppi^{\text{down}})/(2w)$ between times $0$, $t$ and $t+\epsilon$,
 and then separately along $\ppi^{\text{up}}/w$ and $\ppi^{\text{down}}/w$ between times $t$, $t+\epsilon$ and $1$ to get
 \begin{align*}
  \int & \frac{\rho_t^{\text{down}}}{w} \log \frac{\rho_t^{\text{down}}}{w} \,\d\mm\\
    & \le \frac{\epsilon}{t+\epsilon} \int \frac{\rho_0^{\text{down}}}{w} \log \frac{\rho_0^{\text{down}}}{w}\,\d\mm + \frac{t}{t+\epsilon} \int \frac{\rho_{t+\epsilon}^{\text{down}}+\rho_{t+\epsilon}^{\text{up}}}{2w} \log \frac{\rho_{t+\epsilon}^{\text{down}}+\rho_{t+\epsilon}^{\text{up}}}{2w}\,\d\mm + \frac{|K|}{2}\frac{t}{t+\epsilon}l^2\\
    & \le \frac{\epsilon}{t+\epsilon} \log \frac{C}{w} - \frac{t}{t+\epsilon} \log 2 \\
    & \qquad     + \frac{t}{2(t+\epsilon)} \left(\int \frac{\rho_{t+\epsilon}^{\text{down}}}{w} \log \frac{\rho_{t+\epsilon}^{\text{down}}}{w}\,\d\mm + \int \frac{\rho_{t+\epsilon}^{\text{up}}}{w} \log \frac{\rho_{t+\epsilon}^{\text{up}}}{w}\,\d\mm \right )+ \frac{|K|}{2}\frac{t}{t+\epsilon}l^2 \\
    & \le \frac{\epsilon}{t+\epsilon} \log \frac{C}{w} - \frac{t}{t+\epsilon} \log 2 \\
& \qquad     + \frac{t}{2(t+\epsilon)} \left( \frac{\epsilon}{1-t}\int \frac{\rho_{1}^{\text{down}}}{w} \log \frac{\rho_{1}^{\text{down}}}{w}\,\d\mm + \frac{1-t-\epsilon}{1-t}\int \frac{\rho_{t}^{\text{down}}}{w} \log \frac{\rho_{t}^{\text{down}}}{w}\,\d\mm \right ) \\
& \qquad     + \frac{t}{2(t+\epsilon)} \left( \frac{\epsilon}{1-t}\int \frac{\rho_{1}^{\text{up}}}{w} \log \frac{\rho_{1}^{\text{up}}}{w}\,\d\mm + \frac{1-t-\epsilon}{1-t}\int \frac{\rho_{t}^{\text{up}}}{w} \log \frac{\rho_{t}^{\text{up}}}{w}\,\d\mm \right )+ |K|\frac{t}{t+\epsilon}l^2 \\
& \le \left(\frac{\epsilon}{t+\epsilon} + \frac{t\epsilon}{(1-t)(t+\epsilon)} \right)\log \frac{C}{w} - \frac{t}{t+\epsilon} \log 2\\
& \qquad + \frac{t(1-t-\epsilon)}{(t+\epsilon)(1-t)}\int \frac{\rho_t^{\text{down}}}{w} \log \frac{\rho_t^{\text{down}}}{w} \,\d\mm + |K|\frac{t}{t+\epsilon}l^2.
 \end{align*}
 Moving the integral term from the right-hand side to the left and using \eqref{eq:lengthbound} gives
 \begin{align*}
 \frac{\epsilon}{(t+\epsilon)(1-t)}& \int \frac{\rho_t^{\text{down}}}{w} \log \frac{\rho_t^{\text{down}}}{w} \,\d\mm \\
& \le \left(\frac{\epsilon}{t+\epsilon} + \frac{t\epsilon}{(1-t)(t+\epsilon)} \right) \log \frac{C}{w} - \frac{t}{t+\epsilon} \log 2 + |K|\frac{t}{t+\epsilon}l^2  \\
& \le \left(\frac{\epsilon}{t+\epsilon} + \frac{t\epsilon}{(1-t)(t+\epsilon)} \right) \log \frac{C}{w} - \frac{2t}{3(t+\epsilon)} \log 2
 \end{align*} 
 which is the same as
 \begin{equation}\label{eq:almostfinalup}
  \int \frac{\rho_t^{\text{down}}}{w} \log \frac{\rho_t^{\text{down}}}{w} \,\d\mm \le \log \frac{C}{w} - \frac{2t(1-t)}{3\epsilon} \log 2.
 \end{equation}

 Using convexity along $\ppi^{\text{up}}/w$ once more together with \eqref{eq:almostfinalup}, \eqref{eq:lengthbound} and the fact that $\epsilon \le t$ yields
 \begin{align*}
  \int \frac{\rho_{t+\epsilon}^{\text{up}}}{w} \log \frac{\rho_{t+\epsilon}^{\text{up}}}{w}\,\d\mm & \le 
   \frac{\epsilon}{1-t} \log \frac{C}{w} + \frac{1-t-\epsilon}{1-t}\int \frac{\rho_{t}^{\text{down}}}{w} \log \frac{\rho_{t}^{\text{down}}}{w}\,\d\mm + \frac{|K|}2\epsilon(1-t-\epsilon) l^2\\
 & \le \log \frac{C}{w} - \frac{t(1-t-\epsilon)}{3\epsilon}\log 2.
 \end{align*}

 When we combine this with \eqref{eq:lowerboundsmall} we have
 \[
  3\epsilon \left(\log \frac{\epsilon}{10\mm\big(B(x,l/2)\big)} - \log C \right)\le - t(1-t-\epsilon) \log 2 \le - \frac{TS}2\log 2
 \]
 which is a contradiction as the left-hand side goes to $0$ when $\epsilon \downarrow 0$.
\end{proof}

\begin{remark}
 In the proof of Theorem \ref{thm:main} we worked with two fixed marginals $\mu_0, \mu_1$.
 Therefore the assumption in Theorem \ref{thm:main} of being a strong $CD(K,\infty)$-space could be weakened.
 The slightly stronger (but more complicated-looking) version of Theorem \ref{thm:main} would be the following.

 Let $\mu_0, \mu_1 \in \probt X$ be absolutely continuous.
 Suppose that for every $\ppi \in \gopt(\mu_0,\mu_1)$, every $t,s \in [0,1]$ and any Borel function
 $f \colon \geo(X) \to [0,\infty)$ with $\int_{\geo(X)} f\,\d\ppi = 1$ the entropy $\entv$ is $K$-convex along 
 $(\res{s}{t})_\#(f\ppi)$.
 Then any $\ppi \in \gopt(\mu_0,\mu_1)$ is concentrated on a set of non-branching geodesics.
\end{remark}

\section{Proof of Corollary \ref{cor:maps}}

Let us now outline the proof of Corollary \ref{cor:maps}. As usual, it suffices to prove that every optimal plan
is given by a map. This implies uniqueness of optimal plans because for any $\ppi^1,\ppi^2 \in \gopt(\mu_0,\mu_1)$ also $\frac12(\ppi^1 + \ppi^2) \in \gopt(\mu_0,\mu_1)$.
Suppose then that there are two
$\mu_0, \mu_1 \in \probt X$ which are absolutely continuous with respect to $\mm$ and that there is
$\ppi \in \gopt(\mu_0,\mu_1)$ which is not induced by an optimal map.

Because $\mu_0 = \rho_0\mm$ is absolutely continuous, the union
\[
 \bigcup_{C \ge 0}\Gamma_C, \qquad \Gamma_C = \{\gamma \in \geo(X) \,:\, \rho_0(\gamma_0) \le C\}
\]
has full $\ppi$-measure. Therefore for some $C \ge 0$ the measure $\ppi|_{\Gamma_C}$ is not induced by a map.
Therefore we may assume $\mu_0$ to have bounded density. Similarly we may assume $\mu_1$ to have bounded density.
Emptying the space $X$ with larger and larger balls we may also assume $\mu_0$ and $\mu_1$ to have bounded supports.
Therefore we may assume $\entv(\mu_0), \entv(\mu_1) \in \R$.

In the proof of \cite[Theorem 3.3]{G2011} Gigli finds two probability measures $\ppi^1, \ppi^2 \ll \ppi$ with $\ppi^1 \perp \ppi^2$
and $(e_0)_\#\ppi^1 = (e_0)_\#\ppi^2 = \mm|_D$ for a compact $D \subset X$ with $\mm(D)>0$.
But by \cite[Lemma 3.2]{G2011}, which holds also under the strong $CD(K,\infty)$ assumption, we have
\[
 \liminf_{t \searrow 0} \mm(\{\rho_t^i > 0\}) \ge \mm(D), \qquad i =1,2
\]
for the densities $\rho_t^i$ of $(e_t)_\#\ppi^i$. Therefore for some small time $t \in (0,1)$ the sets
$\{\rho_t^1>0\}$ and $\{\rho_t^2>0\}$ must intersect in a set $E$ of positive $\mm$-measure.

So far we have followed the proof of \cite[Theorem 3.3]{G2011}. Now the final step in the case of non-branching
$CD(K,\infty)$-spaces would be to say that no two different geodesics in the support
of an optimal transport can intersect. In the essentially non-branching spaces this conclusion is not so clear,
so we argue differently.

The heuristic idea is to mix the measures $\ppi^1$ and $\ppi^2$ so that at time $t$ we are allowed to change from 
the geodesics where $\ppi^1$ lives to the geodesics where $\ppi^2$ lives, and vice versa.

To write this more rigorously we first pushforward the combined measure $\ppi^1 + \ppi^2$ to the left and right part from time $t$
\[
 \ppi^\text{left} = \frac12\left((\res{0}{t})_\#\ppi^1 + (\res{0}{t})_\#\ppi^2\right), \qquad\ppi^\text{right} = \frac12\left((\res{t}{1})_\#\ppi^1 + (\res{t}{1})_\#\ppi^2\right).
\]
Now let $\{\ppi_x^\text{left}\}$ be the disintegration of $\ppi^\text{left}$ with respect to $e_1$ and 
let $\{\ppi_x^\text{right}\}$ be the disintegration of $\ppi^\text{right}$ with respect to $e_0$.

Observe that the mapping
\[
 \Sp \colon C([0,1];X) \to \left\{(\gamma^1, \gamma^2) \in C([0,1];X) \times C([0,1];X) \,:\, \gamma_1^1 = \gamma_0^2\right\}
  \colon \gamma \mapsto (\res{0}{t}\gamma, \res{t}{1}\gamma)
\]
is bi-Lipschitz.
We define the collection of measures $\{\ppi_x\}$ on $C([0,1];X)$ as pullbacks of the product measures $\ppi_x^\text{left} \times \ppi_x^\text{right}$
under $\Sp$, in other words as the pushforwards
\[
 \ppi_x = (\Sp^{-1})_\#(\ppi_x^\text{left} \times \ppi_x^\text{right}).
\]
Finally we integrate these up to a mixed measure $\ppi^\text{mix}$ by requiring for all Borel $f  \colon C([0,1];X)\to [0,\infty]$
that
\[
 \int_{C([0,1];X)}f(\gamma)\,\d\ppi^\text{mix}(\gamma) = \int_{C([0,1];X)}\int_{(e_t)^{-1}(x)}f(\gamma)\,\d\ppi_x(\gamma)\,\d((e_1)_\#\ppi^\text{left})(x).
\]

Since $\ppi$ is optimal, there exists a set $\Gamma \subset \geo(X)$ such that $\ppi(\Gamma) = 1$ and the set 
$\{(\gamma_0,\gamma_1) \,:\, \gamma \in \Gamma\}$ is cyclically monotone. Because $\ppi^1, \ppi^2 \ll \ppi$,
also $\ppi^1(\Gamma) = \ppi^2(\Gamma) = 1$. For any pair $\gamma^1,\gamma^2 \in \Gamma$
with $\gamma_t^1 = \gamma_t^2$ we have by the cyclical monotonicity and the triangle inequality that
\begin{equation}\label{eq:chain}
\begin{split}
  \sfd^2(\gamma_0^1,\gamma_1^1) + \sfd^2(\gamma_0^2,\gamma_1^2)
 & \le \sfd^2(\gamma_0^1,\gamma_1^2) + \sfd^2(\gamma_0^2,\gamma_1^1)\\
 & \le \left(tl(\gamma^1) + (1-t)l(\gamma^2)\right)^2 + \left(tl(\gamma^2) + (1-t)l(\gamma^1)\right)^2\\
 & = l(\gamma^1)^2 + l(\gamma^2)^2 - 2t(1-t)\left(l(\gamma^1)-l(\gamma^2)\right)^2\\
 & \le l(\gamma^1)^2 + l(\gamma^2)^2 = \sfd^2(\gamma_0^1,\gamma_1^1) + \sfd^2(\gamma_0^2,\gamma_1^2),
\end{split}
\end{equation}
and so all the inequalities in the above chain \eqref{eq:chain} are equalities.
Consequently $l(\gamma^1) = l(\gamma^2)$ meaning that for $\mm$-almost every $x \in X$ there 
exists $l_x$ such that $\ppi_x^\text{left}$ is concentrated
on geodesics of length $tl_x$ and $\ppi_x^\text{right}$ is concentrated on geodesics of length $(1-t)l_x$.
The measure $\ppi^\text{mix}$ is concentrated on 
\[
\left\{\gamma \in C([0,1];X)\,:\, \text{there exist }\gamma^1,\gamma^2 \in \Gamma \text{ s.t. } \res{0}{t}\gamma = \res{0}{t}\gamma^1 \text{ and } \res{t}{1}\gamma = \res{t}{1}\gamma^2\right\},
\]
which by the equalities in \eqref{eq:chain} is a subset of $\geo(X)$.

Furthermore, since $(e_i)_\#\ppi^\text{mix} = (e_i)_\#\left(\frac12(\ppi^1+\ppi^2)\right)$ for $i = 0,1$ and
\begin{align*}
 \int_{\geo(X)} \frac12\sfd^2(\gamma_0,\gamma_1)\,\d(\ppi^1+\ppi^2)(\gamma)
  & = \int_{X} \frac12l_x^2\,\d((e_t)_\#(\ppi^1+\ppi^2))(x)\\
  & = \int_{X} l_x^2\,\d((e_t)_\#\ppi^\text{mix})(x) =\int_{\geo(X)}\sfd^2(\gamma_0,\gamma_1)\,\d\ppi^\text{mix}(\gamma),
\end{align*}
the measure $\ppi^\text{mix}$ is optimal.



Because $\ppi^1 \perp \ppi^2$ and because $\{\rho_t^1>0\}$ and $\{\rho_t^2>0\}$ intersect in the set $E$ of positive $\mm$-measure,
 we know that for $\mm$-almost every $x \in E$ at least one of the measures $\ppi_x^\text{left}, \ppi_x^\text{right}$ is not
a Dirac mass.
Therefore $\ppi^\text{mix}$ is not essentially non-branching or the time-inverse $I_\#\ppi^\text{mix}$ defined via the mapping
\[
 I \colon \geo(X) \to \geo(X) \colon \gamma \mapsto (\gamma' \colon [0,1] \to X \colon t \mapsto \gamma_{1-t})
\]
is not essentially non-branching. This contradicts Theorem \ref{thm:main}.

\begin{acknowledgment}
T.R. acknowledges the financial support of the European Project ERC AdG *GeMeThNES* and the Academy of Finland project no. 137528.
The authors also thank the anonymous referee for carefully reading the paper.
\end{acknowledgment}

\end{document}